\documentclass[a4paper,numbook,babel,final,envcountsame,11pt]{article}

\usepackage{amssymb}
\usepackage{amsthm}
\usepackage{amsmath}
\usepackage{graphicx}
\usepackage{array,hhline}

\newtheorem{lemma}{Lemma}[section]
\newtheorem{theorem}[lemma]{Theorem}
\newtheorem{proposition}[lemma]{Proposition}

\newtheorem{corollary}[lemma]{Corollary}
\theoremstyle{definition}
\newtheorem{definition}[lemma]{Definition}

\numberwithin{equation}{section}
\numberwithin{figure}{section}

\begin{document}


\title{\huge On characterization of poised nodes for a space of bivariate functions}         
\author{Hayk Avdalyan, Hakop Hakopian}        
\date{}          

\maketitle

\begin{abstract}
There are several examples of spaces of univariate functions for which we have a characterization of all sets of knots which are
poised for the interpolation problem. For the standard spaces of univariate polynomials, or spline functions the mentioned results are well-known. In contrast with this there are no such results in the bivariate case. As an exception one may consider only the Pascal classic theorem, in the interpolation theory interpretation. In this paper we consider a space of bivariate piecewise linear functions, for which we can readily find out whether the given node set is poised or not. The main tool we use for this purpose is the reduction by a basic subproblem, introduced in this paper.
\end{abstract}


\section{Introduction and background}
Consider the interpolation problem with a finite dimensional space of univariate functions $S$ and a set of knots $t_1,\ldots,t_m\in \mathbb{R},$ that is
for a given data $\bar c:=\{c_1,\ldots,c_m\}$ find a function $s\in S$ satisfying the conditions
\begin{equation}\label{a1}
 s(t_i)=c_i,\quad i=1,\ldots,m.
\end{equation}
We say that the set of knots is poised for $S$ if for any data $\bar c$ there is a unique function $s\in S$ satisfying the conditions \eqref{a1}.
A necessary condition of the poisedness is
$$m=\dim S.$$

There are several cases of spaces $S$ of univariate functions for which we have a characterization of all poised sets.

Suppose that $S\equiv \pi_n$ is the space of polynomials of degree at most $n.$ We have that $\dim\pi_n=n+1.$ Then, according to the Lagrange theorem, all sets of knots $\{t_1,\ldots,t_{n+1}\}$
are poised.

Now suppose that $S\equiv S_{n,m}:=S_{n;x_0,\ldots,x_m}$ is the space of spline functions of order $n$ and $a:=x_1<\cdots<x_m=:b$ are the knots of the space (see, e.g., \cite{BHS}).
Here $s\in S_{n,m}$ means that $s$ is piecewise polynomial function of degree at most $n-1,$ $s$ vanishes outside of the segment $[a,b]$ and $s$ belongs to the differentiability class $\mathbb{C}^{(n-2)}.$
We have that $\dim S_{n,m}=m-n+1$ and the set of interpolation knots $\{t_1,\ldots,t_{m-n+1}\}$ is poised if and only if the following conditions are satisfied:
$$x_i<t_i<x_{i+n},\quad i=1,\ldots,m-n+1.$$
This result is due to Schoenberg and Whitney \cite{SW1,SW2}.
In the univariate case there are also several characterization results concerning the trigonometric interpolation.

In contrast with this there are no such results in the bivariate case. As an exception one may consider only the Pascal classic theorem, in the interpolation theory interpretation (see, e.g., \cite{HJZ}.)
To present it let $S\equiv \Pi_2$ be the space of bivariate polynomials of total degree at most $2.$
We have that $\dim\Pi_2=6.$ Let also $\Pi_1$ be the space of bivariate linear functions, $\dim\Pi_1=3.$
Then consider any set of $6$ nodes $\{A_1,\ldots,A_6\}$ in the plane. Construct $3$ new nodes as follows:
$$B_1=\ell_{12}\cap\ell_{45},\quad B_2=\ell_{23}\cap\ell_{56}\quad B_3=\ell_{34}\cap\ell_{61},$$
where $\ell_{ij}$ is the line passing through the nodes $A_i$ and $A_j.$  Then according to the Pascal theorem the $6$ nodes $\{A_1,\ldots,A_6\}$ are lying in a conic if and only if the $3$ nodes $\{B_1,B_2,B_3\}$ are collinear. To arrange the case connected with the parallel pairs of lines one may replace the plane with the projective one. The interpolation version of this result is:

\begin{theorem}[Pascal] Any $6$ nodes $\{A_1,\ldots,A_6\}$ in the plane are poised for $\Pi_2$ if and only if the set of respective $3$ nodes $\{B_1,B_2,B_3\}$ is poised for $\Pi_1.$
\end{theorem}
Note that the latter poisedness condition with $\Pi_1$ merely means that $B_1,B_2,B_3$ are not collinear.

Next we are going to introduce the space of bivariate functions we will consider in this paper. For this we need
some preliminaries.

Let us define a strip of triangles.
Fix a sequence of points in the plane
$$\mathcal V :=\left\{V_1,\ldots, V_{n+2}\right\}$$ for the vertices of triangles.
Then consider the $n$ triangles
\begin{equation}\label{tr}
 \Delta_i:=\left[V_i,V_{i+1},V_{i+2}\right],\ i=1,\ldots,n.
\end{equation}
Note that by a triangle  we mean the \emph{closed} set bounded by the sides of the triangle.
The sequence of the triangles
\begin{equation*}\label{btr}
 {{\bar\triangle}}:=\left\{\Delta_1,\ldots, \Delta_n\right\}
\end{equation*} makes a triangulation of the following strip (see Fig. \ref{pic1})
$$|{\bar\triangle}|:=\cup_{i=1}^n\Delta_i.$$ Sometimes it is convenient to call itself $\bar\triangle$ a strip, too.

\begin{figure}[ht] 
\centering
\includegraphics[scale=0.5]{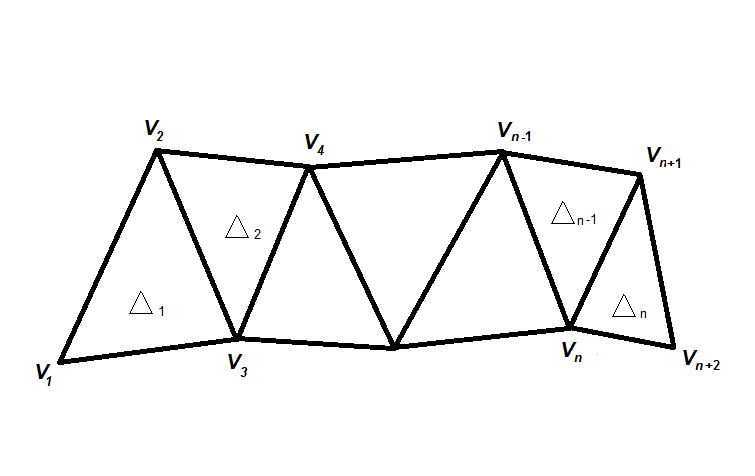}
\caption{The strip $|{\bar\triangle}|$ }\label{pic1}
\end{figure}
Here we require that
the intersection of any neighboring triangles $\Delta_i$ and $\Delta_{i+1},\ i=1,\ldots,n-1,$ is the common side $\left[V_{i+1},V_{i+2}\right],$
any pair of triangles $\Delta_i$ and $\Delta_{i+2},\ i=1,\ldots,n-2,$ have a single common point which is the vertex $V_{i+2},$ and
all the other pairs of the triangles are disjoint.

It is easily seen that the sides $\left[V_i,V_{i+2}\right],\ i=1,\ldots, n,$ together with the sides $\left[V_1,V_2\right]$ and $\left[V_{n+1},V_{n+2}\right],$ form the boundary of the strip $|{\bar\triangle}|$  (see Fig. 1). We call these sides \emph{boundary} sides. The remaining sides of the triangles are called \emph{interior} sides of the strip.
Let us call also the sides $\left[V_1,V_2\right]$ and $\left[V_{n+1},V_{n+2}\right],$ the left and right (boundary) sides of the strip $|{\bar\triangle}|,$ respectively, and denote
$$l({\bar\triangle}):=\left[V_1,V_2\right],\quad r({\bar\triangle}):=\left[V_{n+1},V_{n+2}\right].$$
For a triangle $\Delta_i$ given in \eqref{tr} we call the sides $\left[V_i,V_{i+1}\right]$ and $\left[V_{i+1},V_{i+2}\right],$ the left and right sides of it, respectively.

Denote by $S({\bar\triangle})$ the linear space of continuous piecewise linear functions on ${\bar\triangle}.$
More precisely, $s\in S({\bar\triangle})$ means that

i) $s\vert_{\Delta_i} \in \Pi_1,\ i=1,\ldots,n.$

ii) $s\in C(|\bar\triangle|).$

Here $s\vert_D$ means the restriction of $s$ on $D.$

\begin{definition} A set of nodes ${\mathcal A} :=\left\{A_1,\ldots, A_{m}\right\}$ is called poised for the space $S({\bar\triangle})$ if
for any data $\{c_i,\ i=1,\ldots,m\}$ the interpolation problem
\begin{equation}\label{intpr}
  s(A_i)=c_i,\ i=1,\ldots,m,
\end{equation}
has exactly one solution $s\in S({\bar\triangle}).$
\end{definition}
Let us denote the interpolation problem \eqref{intpr} by $\left\{{\bar\triangle}, \mathcal A\right\}.$

The aim of this paper is the characterization of all poised sets for the space $S({\bar\triangle}).$
The following is a necessary condition of the poisedness:
\begin{equation}\label{dim}
 \#\mathcal A = m=\dim S({\bar\triangle}).
\end{equation}

To prove this consider the fundamental functions of the interpolation $s_i^\star\in S({\bar\triangle}) ,\ i=1,\ldots, m,$ defined by the interpolation conditions
$$s_i^\star(A_j)=\delta_i^j,\ i,j=1,\ldots,m,$$
where $\delta$ is the symbol of Kronecker.

Obviously the fundamental functions are linearly independent and for the solution of the interpolation problem \eqref{intpr} we have the following formula of Lagrange:
$$ s=\sum_{i=1}^{m}c_i s_i^\star.$$
Thus, the fundamental functions form a basis for $S({\bar\triangle})$ and we get \eqref{dim}.

Now, let us show that the set of vertices $\mathcal V$ is a poised set for $S({\bar\triangle}).$
Indeed, having the values of a linear function at the vertices of a triangle $\Delta_{i},\ i=1,\ldots,n,$ we recover it in a unique way on the triangle.
On the other hand it is easily seen that the recovered piecewise linear function is continuous on $|{\bar\triangle}|.$
Thus, the dimension of $S({\bar\triangle})$ equals to the number of the vertices in the strip $|{\bar\triangle}|:$
\begin{equation}\label{dim1}
 \dim S({\bar\triangle})=n+2.
\end{equation}
In view of \eqref{dim}, we obtain that any poised set $\mathcal A$ for the space $S({\bar\triangle})$ consists of $n+2$ nodes:
\begin{equation}\label{A=n+2}
 \#\mathcal A=n+2.
\end{equation}
Let us call interpolation problem $\left\{{\bar\triangle}, \mathcal A\right\}$ satisfying this condition \emph{exact}.\hfill\break
In the cases $\#\mathcal A<n+2$ and $\#\mathcal A>n+2$ we call the problem $\left\{{\bar\triangle}, \mathcal A\right\}$ \emph{underdetermined} and \emph{overdetermined}, respectively.

Denote by $$\sigma_i^\star,\ i=1,\ldots,n+2,$$
the fundamental polynomials with respect to the poised set $\mathcal V.$
They form a basis for the space $S({\bar\triangle}).$
Thus we have the following representation for any piecewise linear function $s\in S({\bar\triangle}):$
$$ s=\sum_{i=1}^{n+2}s(V_i) \sigma_i^\star.$$
In view of this representation the interpolation problem \eqref{intpr} reduces to $m$ linear equations with $n+2$ unknowns, which are the values a piecewise linear function
at vertices $\mathcal V.$
Hence an exact interpolation problem $\left\{{\bar\triangle}, \mathcal A\right\}$ is posed if and only if the following \emph{Vandermonde determinant} does not vanish:
\begin{equation}\label{vdnz}
 V_{\left\{{\bar\triangle}, \mathcal A\right\}}\neq 0,
\end{equation}
where
$$
V_{\left\{{\bar\triangle}, \mathcal A\right\}} = \left|\hspace{-.3cm}\begin{matrix}&  \sigma_1^\star(A_1)& \cdots &\sigma_1^\star(A_{n+2})\\
& \sigma_2^\star(A_1) & \cdots &\sigma_2^\star(A_{n+2})\\
& &\vdots &\\
& \sigma_{n+2}^\star(A_1) & \cdots &\sigma_{n+2}^\star(A_{n+2})\\
\end{matrix}\right|.
$$
Thus our main problem can be formulated in terms of Vandermonde determinant in the following way:

\noindent Characterize all exact sets $\mathcal A$ for which the Vandermonde determinant does not vanish, i.e., \eqref{vdnz} holds.

The following two propositions are basic Linear Algebra facts.
\begin{proposition}\label{0p1}Given an exact problem $\left\{{\bar\triangle}, \mathcal A\right\}.$ Then each of the following conditions is equivalent to the poisedness of $\left\{{\bar\triangle}, \mathcal A\right\}$:\\
i) All fundamental functions $s_i^\star\in {\bar\triangle} ,\ i=1,\ldots, n+2,$ exist.\\
ii) $s\in S({\bar\triangle}),\ s\vert_\mathcal A=0 \Rightarrow s=0.$
\end{proposition}

\begin{proposition}\label{0p2}
Given a problem $\left\{{\bar\triangle}, \mathcal A\right\}.$ Then the following hold:\\
i) If the problem is underdetermined then there is a function $s\in S({\bar\triangle})$ such that $s\vert_\mathcal A=0, \ \hbox{and}\ s\neq 0.$\\
ii) If the problem is overderdetermined then there is a node in $\mathcal A$ for which no fundamental function exists.
\end{proposition}

\section{Subproblems}
For a given strip ${\bar\triangle}$ denote by ${{\bar\triangle}}_k^m,\ \le k\le m\le n,$ the following part of it
$${{\bar\triangle}}_k^m:=\left\{\Delta_k,\Delta_{k+1},\ldots, \Delta_m\right\}.$$
For a problem $\left\{{\mathcal A}, {\bar\triangle}\right\},$ denote by $\left\{{\mathcal A}, {\bar\triangle}\right\}_{k}^{m},$
the subproblem with the function space
 $S({\bar\triangle}_k^m),$ and the set of nodes ${\mathcal A}_k^m=\mathcal A \cap {\bar\triangle}_k^m,$ i.e.,
$$\left\{{\mathcal A}, {\bar\triangle}\right\}_{k}^{m}:=\left\{{\mathcal A}_{k}^{m}, {\bar\triangle}_{k}^{m}\right\}.
$$

\subsection{Problems with boundary conditions}

Denote by $S(:\hspace{-.13cm}{{\bar\triangle}})$ the linear space of continuous piecewise linear functions on ${\bar\triangle}$ vanishing at the left (boundary) side of the strip $|{\bar\triangle}|$, i.e.,
$$S(:\hspace{-.13cm}{{\bar\triangle}})=\left\{s\in S({{\bar\triangle}}),\  s\vert_{l(\Delta)}=0\right\}.$$
In the similar way one can define the space of functions vanishing at the right side of the strip $|{\bar\triangle}|,$ i.e.,
$$S({{\bar\triangle}}\hspace{-.135cm}:)=\left\{s\in S({{\bar\triangle}}),\  s\vert_{r(\Delta)}=0\right\}.$$
One may define also the space of functions vanishing at the both left and right sides of the strip $|{\bar\triangle}|,$ i.e.,
$$S(:\hspace{-.13cm}{{\bar\triangle}}\hspace{-.135cm}:)=\left\{s\in S({{\bar\triangle}}),\   s\vert_{l(\Delta)}=s\vert_{r(\Delta)}=0\right\}.$$

By counting the number of vertices of the strip where a function from the space may differ
from $0,$ we get, in the same way as in the proof of \eqref{dim1}, that
$$\dim S(:\hspace{-.13cm}{{\bar\triangle}})=\dim S({{\bar\triangle}}\hspace{-.135cm}:)=n\ \hbox{and}\ \dim S(:\hspace{-.13cm}{{\bar\triangle}}\hspace{-.135cm}:)=n-2.$$

Denote the interpolation problems with the spaces $S(:\hspace{-.13cm}{{\bar\triangle}}), S({{\bar\triangle}}\hspace{-.135cm}:),$  $S(:\hspace{-.13cm}{{\bar\triangle}}\hspace{-.135cm}:),$ and a node set $\mathcal A$ by $\left\{{{\bar\triangle}}\hspace{-.135cm}:, \mathcal A\right\},$ $\left\{:\hspace{-.13cm}{{\bar\triangle}}, \mathcal A\right\},$ $\left\{:\hspace{-.13cm}{{\bar\triangle}}\hspace{-.135cm}:, \mathcal A\right\},$ respectively.

As we will see below the poisedness of interpolation problems with boundary conditions readily can be reduced to
the previous general interpolation problems by just adding two nodes in the left or/and right sides of the strip.

For this purpose it is convenient to use the following notation for the strip ${\bar\triangle}$

$$2+\mathcal A:=\mathcal A\cup \{ V_1,V_2\},$$
where $V_1$ and $V_2$ are the two vertices of $l(\bar\triangle).$

$$\mathcal A+2:=\mathcal A\cup \{ V_{n+1},V_{n+2}\},$$
where $V_{n+1}$ and $V_{n+2}$ are the two vertices of $r(\bar\triangle).$ Denote also

$$2+\mathcal A+2:=\mathcal A\cup \{ V_1,V_2, V_{n+1},V_{n+2}\}.$$
Let us call two interpolation problems \emph{equivalent} if both they are poised or both they are not poised.
By using Proposition \ref{0p1}, ii), we readily get
\begin{proposition}\label{0p3}
The folllwing pairs of interpolation problems are equivalent:\\
\centerline{$\left\{{{\bar\triangle}}\hspace{-.135cm}:, \mathcal A\right\}$ and $\left\{{{\bar\triangle}}, \mathcal A+2\right\},$}\\
\centerline{$\left\{:\hspace{-.13cm}{{\bar\triangle}}, \mathcal A\right\}$ and $\left\{{{\bar\triangle}}, 2+\mathcal A\right\},$}\\
\centerline{$\left\{:\hspace{-.13cm}{{\bar\triangle}}\hspace{-.135cm}:, \mathcal A\right\}$ and $\left\{{{\bar\triangle}},2+ \mathcal A+2\right\}.$}
\end{proposition}

Note that in the subproblem $\left\{{\bar\Delta,\mathcal A}+2\right\}_{k}^m$ we have that
$$\mathcal A+2:=\mathcal A\cup \{ V_{m+1},V_{m+2}\},$$
where $V_{m+1}$ and $V_{m+2}$ are the two vertices of $r({\bar\triangle}_k^m).$
The situation is similar with the sets $2+\mathcal A$ and $2+\mathcal A +2.$

\section{Reductions of interpolation problems}

Below we bring a main theorem regarding the reduction of an interpolation problem having an exact or overdetermined subproblem.
\begin{theorem} \label{main} Suppose that $\left\{{\bar\triangle},\mathcal A\right\}$ is an exact problem where the strip $\bar\triangle$ consists of $n$ triangles and $\left\{\bar\triangle,\mathcal A\right\}_k^m$ is a subproblem, where $1\le k\le m\le n$.\\
 Then the following hold.\\
i) If the subproblem $\left\{\bar\triangle,\mathcal A\right\}_k^m$ is exact and not poised, or overdetermined, then the problem  $\left\{\bar\triangle,\mathcal A\right\}$ is not poised.\\
ii) If the subproblem $\left\{\bar\triangle,\mathcal A\right\}_k^m$ is poised then the problem $\left\{{{\bar\triangle}},\mathcal A\right\}$ is poised if and only if
the both following two reduced problems
\begin{equation}\label{0r}
\left\{{{\bar\triangle},\mathcal A}+2\right\}_{1}^{k-1},\quad \left\{{{\bar\triangle}},2+{\mathcal A}\right\}_{m+1}^n,
\end{equation}
are exact and poised.
\end{theorem}
Note that in the cases $k=1$ and $m=n$ we have just one reduced problem instead of two.
\begin{proof}
Suppose first that the subproblem $\left\{{{\bar\triangle}},\mathcal A\right\}_k^m$ is exact and not poised, or overdetermined. Then, in view of Propositions \ref{0p1}, i), and \ref{0p2}, ii), there is a node $A\in{\mathcal A}_k^m$ which does not have a fundamental function. Then, evidently the same node $A$ does not  have a fundamental function for the whole node set $\mathcal A,$ too. Hence, in view of Proposition \ref{0p1}, i), we get that the problem $\left\{{{\bar\triangle}},\mathcal A\right\}$ is not poised.

Now consider the case when the subproblem $\left\{{{\bar\triangle}},\mathcal A\right\}_k^m$ is poised.

Let us assume that the both reduced problems in \eqref{0r} are poised. Then let us prove that the problem $\left\{{{\bar\triangle}},\mathcal A\right\}$ is poised, too. Notice first that the problem $\left\{{{\bar\triangle}},\mathcal A\right\}$ is exact.
Thus, by following Proposition \ref{0p1}, ii), assume that $s\in S({\bar\triangle}),\ s\vert_{\mathcal A}=0.$
Now, since the subproblem $\left\{{{\bar\triangle}},\mathcal A\right\}_k^m$ is poised, we conclude that $s$ vanishes on the triangles $\Delta_i,\ i=k,k+1,\ldots,m.$
Therefore $s$ vanishes at the right side of the triangle $\Delta_{k-1}$ and at the left side of the triangle $\Delta_{m+1}.$
Thus we have that $s\vert_{\left\{{\mathcal A}_{1}^{k-1}+2\right\}}=0$ and $s\vert_{\left\{2+{\mathcal A}_{m+1}^n\right\}}=0.$
Finally, since the problems in \eqref{0r} are poised, we obtain that $s$ vanishes on the triangles $\Delta_i,\ i=1,\ldots,k-1$ and
$\Delta_i,\ i=m+1,m+2,\ldots,n.$ Hence $s$ vanishes on all triangles of ${\bar\triangle}.$

Next let us assume that the problem $\left\{{{\bar\triangle}},\mathcal A\right\}$ is poised and prove that the both reduced problems in \eqref{0r} are poised, too.
Let us show first that these reduced problems are exact.
There are $k-1$ and $n-m$ triangles in the reduced problems \eqref{0r}, respectively.
Thus for exactness we need $k+1$ and $n-m+2$ nodes, respectively, altogether $n+k-m+3$ nodes.
It is easily seen that indeed, in two mentioned problems together we have that many nodes. Indeed, this is the number of the nodes in $\mathcal A$ minus the number of the nodes in ${\mathcal A}_k^m$
and plus $4,$ i.e., $n+2-(m-k+3)+4=n+k-m+3.$ Latter $4\ (=2+2)$ nodes come from the added nodes in the boundary in the node sets $\mathcal A+2$ and $2+\mathcal A.$

Now, assume by way of contradiction that
 a subproblem in \eqref{0r} is not exact. Therefore one of the subproblems, say the first problem in  \eqref{0r}, is underdetermined, and another is overdetermined.
Then, in view of Proposition \ref{0p2}, i), there is a function $s\in S({\bar\triangle}_{1}^{k-1})$ such that
\begin{equation}\label{200}
 s\vert_{\left\{{\mathcal A}_{1}^{k-1}+2\right\}}=0, \ \hbox{and}\ s\neq 0.
\end{equation}
Thus $s$ here vanishes on ${\mathcal A}_{1}^{k-1}$ and, since of "$+2,$" also on the right side of the triangle $\Delta_{k-1}.$
Now let us extend this function $s$ from ${\bar\triangle}_{1}^{k-1}$ till the whole strip ${|\bar\triangle}|$ by defining it to be $0$ on the triangles $\Delta_i,\ i=k, k+1,\ldots,n.$ Denote by $\tilde s$ the extended function.
Then it is easily seen that $\tilde s\in S({\bar\triangle}),\ \tilde s\vert_{\mathcal A}=0$ and $\tilde s \neq 0.$
This, in view of Proposition \ref{0p1}, ii), means that the problem $\left\{{{\bar\triangle}},\mathcal A\right\}$ is not poised, which contradicts our assumption.

Finally, let us show that the both reduced problems in \eqref{0r} are poised.
Assume by way of contradiction that one of them, say the first, is not poised.
Since it is exact, we can use Proposition \ref{0p1}, ii), to get that there is a function $s\in S({\bar\triangle}_{1}^{k-1})$ such that
the relation \eqref{200} is satisfied. From here we continue in the same way as in the above step, after the relation \eqref{200}.
\end{proof}

\section{The basic interpolation problems}

Let us denote by "$3$" the problem  with a strip consisting of a triangle and at least three nodes inside.
The respective \emph{basic} problem, denoted briefly by "$3/B$", is a "$3$" problem with exactly three nodes, which are non-collinear.

Then, let us denote by "$2+2$" the problem  with a strip consisting of two triangles and exactly two nodes in each of them.
Note that in this case no node can lie in the interior side of the strip.
We call a problem "$2+2$" basic and denote it by "$2+2/B$" if in each triangle the line passing through the two nodes there does not intersect the other triangle.
Note that
in the case of "$2+2/B$" problem no node can coincide with a vertex of the strip.

Next consider an interpolation problem denoted by "$2+1+\cdots+1+2,m,$" where $m$ is the number of $1'$s. This is the interpolation problem with a strip  consisting of $m+2$ triangles such that each of the first and the last triangles contains exactly two nodes and each of the other $m$ triangles contains exactly a node. Note that in this case no node can be located in an interior side of the strip.
We call a problem "$2+1+\cdots+1+2,m$" basic and denote it by "$2+1+\cdots+1+2,m/B$" if the line passing through the two nodes in each of the first and the last triangles does not intersect the neighboring triangle.

Note that the "$2+2$" problem can be considered as a special case of the "$2+1+\cdots+1+2,m$" problem, where $m=0.$

\subsection{The poisedness of the basic problems}

Let us start with a simple lemma.
Suppose that two nodes $A,B$ of a node set $\mathcal A$ belong to a triangle of ${\bar\triangle}.$ Denote by $A',B'$ the points of intersection of the line passing through $A$ and $B$ with the sides of the triangle. Let us call $\{A',B'\}$ the intersection pair of $\{A,B\}.$  Denote by ${\mathcal A}'$ the set of the nodes received from $\mathcal A$ by replacing the pair $\{A,B\}$ with $\{A',B'\}$ there.
We call the set $\mathcal A'$ the \emph{line transformation} of the set $\mathcal A$ with respect to the pair $\{A,B\}.$

Two node sets $\mathcal A$ and $\mathcal B$ are called \emph{line-equivalent} if one of them can be obtained from the other by means of several line transformations.

\begin{lemma}\label{2l}
 The interpolation problems $\left\{{\bar\triangle}, \mathcal A\right\}$ and  $\left\{{\bar\triangle}, {\mathcal B}\right\}$ are equivalent, if the node sets $\mathcal A$ and $\mathcal B$ are line-equivalent.
\end{lemma}
\begin{proof}
In view of Proposition \ref{0p1}, ii), it suffices to verify that
\begin{equation}\label{equi}
s\in S({\bar\triangle}),\ s\vert_\mathcal A =0 \Leftrightarrow s\in S({\bar\triangle}),\ s\vert_{{\mathcal A}'} =0,
\end{equation}
where the node set $\mathcal A'$ is the line transformation of the set $\mathcal A$ with respect to an appropriate pair of nodes $\{A,B\}\subset \mathcal A.$
Now notice that \eqref{equi} readily follows from the evident fact that $$s(A)=s(B)=0 \Leftrightarrow s(A')=s(B')=0,$$ where $\{A',B'\}$ is the intersection pair of $\{A,B\}.$ Indeed, each side of this equivalence means merely that $s$ vanishes on the line passing through $A$ and $B.$
\end{proof}

The proof of the following proposition contains an easy algorithm for verifying the poisedness of basic problems.
\begin{proposition} \label{bp} The problems "$3/B$" and "$2+1+2/B$" are always poised.\\
The problem "$2+1+\cdots+1+2,m/B,$" where $m\neq 1,$ may be poised or not, depending on the
configuration of nodes.
Moreover, if the problem is not poised then it becomes poised after changing the location of one of the two nodes in the first or the last triangle, such that the slope of the line passing through the two nodes is changed.
\end{proposition}
Note that the case $m=0$ here concerns the problem  "$2+2/B.$"

\begin{proof} The case of "$3/B$" is obvious.

Consider the problem "$2+1+2/B$" (see Fig. \ref{pic2}).
\begin{figure}[ht] 
\centering
\includegraphics[scale=0.5]{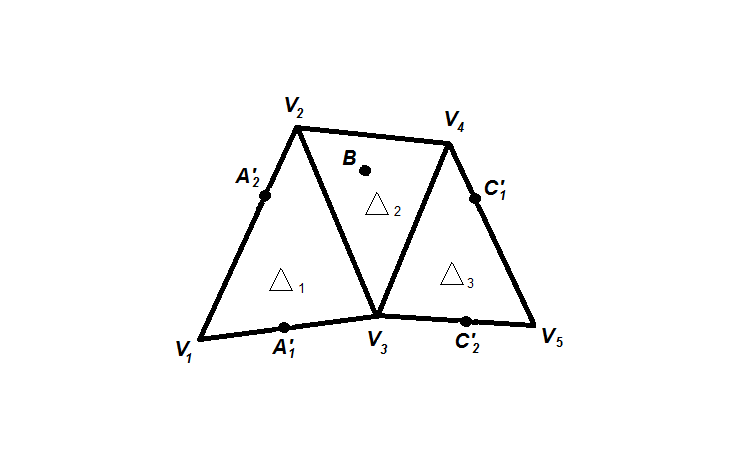}
\caption{The case of "$2+1+2/B$" }\label{pic2}
\end{figure}
Suppose that ${\bar\triangle}=\{\Delta_1,\Delta_2, \Delta_3\}$ and $\mathcal A=\{A_1,A_2,B,C_1,C_2\}$, where $A_1,A_2\in \Delta_1, B\in \Delta_2,$ $C_1,C_2\in \Delta_3.$ In view of Lemma \ref{2l} we may replace the node set $\mathcal A$ with the set $\mathcal A'=\{A_1',A_2',B,C_1',C_2'\},$ where $\{A_1',A_2'\}$ and $\{C_1',C_2'\}$ are the intersection pairs of the respective pairs from $\mathcal A.$ Since the problem $\left\{{{\bar\triangle}},\mathcal A\right\}$ is basic we have that the nodes $A_1',A_2',C_1',C_2'$ belong to the four sides of the triangles $\Delta_1$ and $\Delta_3,$ which are boundary sides of the strip ${\bar\triangle},$ one node in each side (see Fig. 2).
Now suppose by way of contradiction that the problem $\left\{{\bar\triangle}, \mathcal A'\right\}$ is not poised. This in view of Proposition \ref{0p1}, ii), means that there is a function $s\in S({\bar\triangle})$ such that $s\vert_{\mathcal A'}=0$ while $s\neq 0.$
Then it is easily seen that $s\vert_{\Delta_1}\neq 0$ and therefore $s(V_1)\neq 0.$
Assume without loss of generality that $s(V_1)> 0.$ Then, since $s$ is linear and $s(A_1')=s(A_2')=0$ we get readily that
$s(V_2)<0$ and $s(V_3)<0.$ Now, in view of the latter condition and $s(C_1')=s(C_2')=0,$ we get readily that
$s(V_5)>0$ and $s(V_4)<0.$ Thus $s$ assumes negative values at all the vertices of the triangle $\Delta_2$ and it vanishes at $B\in\Delta_2,$
which is a contradiction.

Finally, consider the problem "$2+1+\cdots+1+2,m/B,$" where $m\neq 1.$
Suppose that ${\bar\triangle}=\{\Delta_1,\ldots, \Delta_{m+2}\}$ and $\mathcal A=\{A_1,A_2,B_1,\ldots, B_m,C_1,C_2\}$, where $A_1,A_2\in \Delta_1, B_i\in \Delta_{i+1}, C_1,C_2\in \Delta_{m+2}.$

Assume that the problem $\left\{{\bar\triangle}, \mathcal A\right\}$ is not poised. This in view of Proposition \ref{0p1}, ii), means that there is a function $s$ such that
\begin{equation}\label{contr}
s\in S({\bar\triangle}),\ s\vert_\mathcal A=0 \ \hbox{and}\ s\neq 0.
\end{equation}
It is easily seen that $s\vert_{\Delta_1}\neq 0.$ Indeed, $s\vert_{\Delta_1} = 0$ implies that $s$ vanishes at the left side of $\Delta_2.$ Then, in view of the condition $s(B_1)=0,$ where $B_1\in\Delta_2$ we conclude that $s\vert_{\Delta_2} = 0.$ Notice that $B_1$ does not belong to the left (or right) side of the triangle, since the problem is basic. Continuing this way we
obtain $s\vert_{\Delta_i} = 0,\ i=2,\ldots, m+2$  and thus $s=0,$ contradicting our assumption.
In the same way we get that
\begin{equation}\label{m+2}
s\vert_{\Delta_{m+2}}\neq 0.
\end{equation}
By replacing $s$, if necessary, with a nonzero constant multiple of it, assume that $s(V_1)=1.$ Now, in view of the condition $s(A_1)=s(A_2)=0,$ where $A_i\in\Delta_1,$ the function $s\in S({\bar\triangle})$ is determined on $\Delta_1.$ Indeed, the nodes $V_1,A_1,A_2$ are not collinear since the problem is basic. Hence $s$ is determined at the left side of the second triangle $\Delta_2.$ Next, we have that $s(B_1)=0,$ where $B_1\in\Delta_2$ does not belong to the left (or right) side of the triangle. Thus we get that $s$ is determined on $\Delta_2.$ Continuing this way we
get that $s$ is determined on $\Delta_{m+1}.$ Now $s$ is determined at the left side of the last triangle $\Delta_{m+2}.$ By using the condition $s(C_1)=0,$ where $C_1\in\Delta_{m+2},$ we conclude, in view of \eqref{m+2}, that the zero set of $s$ in $\Delta_{m+2}$ is a line $\ell$ passing through $C_1.$ Thus if the last node $C_2$ lies in this line then we have that $s$ vanishes at $C_2$ too. Hence the condition \eqref{contr} holds, which, in view of Proposition \ref{0p1}, ii), means that the problem is not poised. While the condition $C_2\notin \ell$ implies that $s$ vanishes on $\Delta_{m+2}$ which contradicts \eqref{m+2} and whence \eqref{contr}. Thus the problem is poised in this case.
This consideration makes clear also the "moreover" part of Proposition.\end{proof}

It is worth mentioning that if the restriction of $s$ on the left side of triangle $\Delta_{m+2}$ has a zero (as it will happen in the case of "$2+1+2/B$") denoted by $C$ then the basic problem
is poised, wherever the two last nodes $C_1$ and $C_2$ are situated. Indeed, otherwise we would have that the zero-line $\ell$ of $s$ in the last triangle
passes through $C_1, C_2,$ and $C,$ which contradicts the fact that the problem is basic.

\section{Reductions by basic subproblems}
\subsection{Reduction by "$3/B$"}
In the case of basic subproblem "$3/B$" we have a poised subproblem with one triangle. Thus we get from Theorem \ref{main}, ii), the following
\begin{corollary} \label{3B} Assume that $\left\{{{\bar\triangle}},\mathcal A\right\}$ is an exact problem, where the strip consists of $n$ triangles. Assume also that a triangle $\Delta_k,\ 1\le k\le n,$ contains exactly $3$ non-collinear nodes.
Then the interpolation problem $\left\{{{\bar\triangle}},{\mathcal A}\right\},$ is poised if and only if
the both following two reduced problems
\begin{equation}\label{1r}
\left\{{{\bar\triangle},\mathcal A}+2\right\}_{1}^{k-1},\quad \left\{{{\bar\triangle}},2+{\mathcal A}\right\}_{k+1}^n,
\end{equation}
are exact and poised.
\end{corollary}
Note that in the cases $k=1$ and $k=n$ we have just one reduced problem instead of two.

\subsection{Reduction by  "$2+1+\cdots+1+2/B$"}
For this case we get from Theorem \ref{main} the following
\begin{corollary} \label{2112B} Assume that $\left\{{{\bar\triangle}},\mathcal A\right\}$ is an exact problem, where the strip consists of $n$ triangles. Assume also that the subproblem $\left\{{\bar\triangle},\mathcal A\right\}_k^{k+m+2},\ 1\le k\le n-m-2,$ is basic of type "$2+1+\cdots+1+2,m/B$," where  $m\ge 0, m\neq 1.$
Then the problem is not poised if the subproblem is not poised.
In the case when the subproblem is poised the problem $\left\{{{\bar\triangle}},\mathcal A\right\}$ is poised if and only if
the both following two reduced problems
\begin{equation}\label{2r}
\left\{{{\bar\triangle},\mathcal A}+2\right\}_{1}^{k-1},\quad \left\{{{\bar\triangle}},2+{\mathcal A}\right\}_{k+m+3}^n,
\end{equation}
are exact and poised.
\end{corollary}
Note that in the cases $k=1$ and $k=n-m-2$ we have just one reduced problem instead of two.

Since, by Proposition \ref{bp}, the basic problems of type "$2+1+2/B$" are always poised we get from Theorem \ref{main}, ii), the following
\begin{corollary}\label{212B} Assume that $\left\{{{\bar\triangle}},\mathcal A\right\}$ is an exact problem, where the strip consists of $n$ triangles.
Assume also that the subproblem $\left\{{\bar\triangle},\mathcal A\right\}_k^{k+2},\ 1\le k\le n-2,$ is basic of type "$2+1+2/B$".
Then the problem $\left\{{{\bar\triangle}},\mathcal A\right\}$ is poised if and only if
the both following two reduced problems $$\left\{{{\bar\triangle},\mathcal A}+2\right\}_{1}^{k-1},\quad \left\{{{\bar\triangle}},2+{\mathcal A}\right\}_{k+3}^n,$$ are exact and poised.
\end{corollary}
Note that in the cases $k=1$ and $k=n-2$ we have just one reduced problem instead of two.

\section{Existence of a reduction}

In this section we prove that every exact problem has a subproblem of type "$3$", "2+2", or "$2+1+\cdots+1+2,m$". Next, we show that by applying line-transformations to this subproblem, we can reduce it to a basic subproblem or determine that the given exact problem is not poised. After this, by following the algorithm pointed out in the proof of Proposition \ref{bp}, we may verify whether the basic subproblem is poised or not. If not then, in view of Theorem \ref{main}, i), we conclude that the given exact problem is not poised either. If the basic subproblem is poised then we may apply the respective reduction, given by Corollaries \ref{3B}, \ref{2112B}, or \ref{212B}. Thus we have a complete solution of the poisedness problem.

\begin{theorem} Assume that $\left\{{{\bar\triangle}},\mathcal A\right\}$ is an exact problem.
Then it has a subproblem of type "$3$", "2+2", or "$2+1+\cdots+1+2,m$". Moreover, by using line-transformations, we either determine that the problem is not poised, or we reduce the node set ${\mathcal A}$ to ${\mathcal B}$ such that the line-equivalent problem $\left\{{{\bar\triangle}},{\mathcal B}\right\}$ has a basic subproblem of type "$3/B$", "$2+2/B$", or "$2+1+\cdots+1+2,m/B$".
\end{theorem}

\begin{proof} Suppose that a triangle contains more than $3$ nodes, or exactly $3$ collinear nodes. Then the problem obviously is not poised and also we have a type "$3$" subproblem. If a triangle contains exactly $3$ non-collinear nodes then we have a type "$3/B$" basic subproblem.

Now, let us assume that no triangle of ${\bar\triangle}$ contains more than $2$ nodes.

\emph {Step 1.} Let us verify that the problem $\left\{{\bar\triangle}, \mathcal A\right\}$
contains a subproblem of type  "2+2", or "$2+1+\cdots+1+2,m$".

Let us throw away those triangles of ${\bar\triangle}$ which do not contain nodes from $\mathcal A.$ Let us throw away also those triangles of ${\bar\triangle}$ which contain just one node from $\mathcal A$ located in an interior side of ${\bar\triangle}.$
By saying a node is thrown we mean that it does not belong to the remained triangles.

Notice that only a node located in an interior side of the strip can be thrown. Moreover, this can happen only if the both neighboring triangles were thrown, i.e., if it is the only node in the interior side and also the only node in the both neighboring triangles.

Assume that after this operation the connected blocks of remained triangles in $\bar\triangle$ are
\begin{equation}\label{bl}
{\bar\triangle}_{k_1}^{{k_1}'}, {\bar\triangle}_{k_2}^{{k_2}'},\ \ldots,\ {\bar\triangle}_{k_l}^{{k_l}'},
\end{equation}
where $1\le k_1\le k_1'<k_2\le k_2'<\cdots < k_l\le k_l'\le n.$

Now, let us verify that for a block here the number of nodes (from $\mathcal A$) is greater or equal to the number of triangles plus two.

Assume by way of contradiction that for all blocks in \eqref{bl} the number of nodes does not exceed the number of triangles plus one.
Then we get that the number of nodes in all above connected blocks is less than or equal to the number of triangles in all connected pieces plus $l$, where
$l$ is the number of blocks in \eqref{bl}.

Next, consider the breaks, i.e., the dropped blocks. For each break we have that the number of nodes, i.e., number of the thrown nodes in the above mentioned operation, is less than or equal to the number of triangles there minus $1.$
Indeed, a such node (in an interior side of ${\bar\triangle}$) is thrown if both
triangles containing it where thrown. Thus we can assign two thrown triangles to each thrown node. Also note that all the triangles assigned are different.
Hence the number of nodes in all above mentioned breaks is less than or equal to the number of the triangles there minus $l-1$, since there are at least $l-1$ breaks.

Thus we may conclude that the number of nodes in $\bar\triangle$ does not exceed the number of triangles there plus $1 (=l-(l-1)).$
This contradicts the assumption that the problem $\left\{{{\bar\triangle}},\mathcal A\right\}$ is exact.

Now assume, without loss of generality, that in the first block ${\bar\triangle}_{k_1}^{{k_1}'}$ the number of nodes from $\mathcal A$ is greater or equal to the number of triangles there plus two.
Note that if a triangle in this block contains only one node then it is not located in its left or right side, otherwise the triangle would be thrown away before.

Since no triangle contains three nodes, we get that there is a triangle in this block containing two nodes.
Consider the first such triangle $\Delta_k,$ i.e., a such triangle with the minimal subscript $k\ge i_1.$ If both the two nodes here are lying in the right side of the triangle then the next triangle $\Delta_{k+1}$  contains only these two nodes. Now notice that in the triangles from $\Delta_{i_1}$ till
$\Delta_{k+1},$ the number of nodes equals to $k-i_1+2,$ i.e., the number of triangles here. This means that
there is another triangle in the first block, following $\Delta_{k+1},$ containing two nodes, such that at least one node is not lying in the right side of the triangle.

Then consider the first such triangle $\Delta_{k'}.$ If one of the two nodes here is in the right side of the triangle then the next triangle $\Delta_{k'+1}$ also contains two nodes, because otherwise it would be thrown away before. Now notice that in the triangles from $\Delta_{i_1}$ till
$\Delta_{k'+1},$ the number of nodes does not exceeds $k'-i_1+3,$ i.e., the number of triangles here plus $1.$ This means that
there is another triangle in this series, denoted by $\Delta_{l},$  following $\Delta_{k'+1},$ containing two nodes, which do not lie in its right side. 
This will be the first triangle of the desired subproblem.

Now notice that in the triangles from $\Delta_{i_1}$ till
$\Delta_{l},$ the number of nodes equals to $l-i_1+2,$ i.e., the number of triangles here plus $1.$ This means that
there is another triangle in the first block, following $\Delta_{l},$ containing two nodes.
Denote by $\Delta_{l'}$ the first such triangle. This will be the last triangle of the desired subproblem. First notice that each triangle between $\Delta_{l}$ and $\Delta_{l'},$ if there is a such, contains just one node, which, as was mentioned earlier, cannot be located in its left or right side. Therefore no node in  $\Delta_{l'},$ belongs to the left side of the triangle, since otherwise $\Delta_{l'-1},$ would contain a node in its right side. Note also that in this case $\Delta_{l'-1}$ does not coincide with $\Delta_{l},$ since the latter has no node in its
right side.

 Thus the subproblem $\left\{{\bar\triangle}, \mathcal A\right\}_{l}^{l'}$ is of type "2+2", or "$2+1+\cdots+1+2,m$" with $m\ge 1.$

\emph{Step 2.} Next, by using line-transformations, we either reduce  this subproblem to a basic subproblem or determine that the problem $\left\{{{\bar\triangle}},\mathcal A\right\}$ is not poised.

First suppose that the subproblem $\left\{{\bar\triangle}, \mathcal A\right\}_{l}^{l'}$ is of type "$2+2$".
If it is not basic then the line through the two nodes in a triangle intersects its interior side. Then by replacing these two nodes with their intersection pair we will have in the other triangle one more node, i.e., three nodes. This case was considered in the beginning of the proof.

Finally, suppose that the subproblem $\left\{{\bar\triangle}, \mathcal A\right\}_{l}^{l'}$ is of type "$2+1+\cdots+1+2,m$," with $m\ge 1.$
Now, if this subproblem is not basic then in the same way as in the previous case we may reduce it by line transformation to a problem of the same type, where the number of $1'$s equals to $m-1.$ Thus we may complete readily the proof by using induction on $m,$ where
the first step of induction corresponds to the case $m=0$ considered already.
 \end{proof}
\section{\label{7} Final remarks}
\subsection{Some necessary conditions of poisedness}
Consider a problem $\left\{{{\bar\triangle}},\mathcal A\right\},$ where the strip consists of $n$ triangles.
Denote the number of nodes from $\mathcal A$ in the triangles $\Delta_k, \Delta_{k+1},\ldots, \Delta_m,\ 1\le k\le m \le n,$ by $\nu_k^m:$
$$\nu_k^m:=\#{\mathcal A}_k^m.$$
The following proposition gives some necessary conditions of poisedness.
\begin{proposition}\label{nc}
Given a poised problem $\left\{{{\bar\triangle}},\mathcal A\right\},$ where the strip consists of $n$ triangles. Then for each $k$ and $m,\ 1\le k\le m \le n,$ we have that
\begin{equation}\label{abc3}m-k-1\le \nu_k^m\le m-k+3.\end{equation}
Moreover, in the cases $k=1$ and $m=n$ we have stricter inequalities:
\begin{equation}\label{abc1}m\le \nu_1^m\le m+2,\quad
n-k+1\le \nu_k^n\le n-k+3.\end{equation}
\end{proposition}

\begin{proof} Let us prove first the right inequality in \eqref{abc3}:
\begin{equation}\label{abc4} \nu_k^m\le m-k+3.\end{equation}
Indeed, suppose conversely that $\nu_k^m\ge m-k+4.$ Then the subproblem $\left\{{{\bar\triangle}},\mathcal A\right\}_k^m,$
is overdetermined.
Therefore the problem $\left\{{{\bar\triangle}},\mathcal A\right\},$ in view of Theorem \ref{main}, is not poised, which is a contradiction.

In particular, we get from \eqref{abc4} that
\begin{equation}\label{abc5} \nu_{m+1}^n\le n-m+2,\quad
\nu_1^{k-1}\le k+1.\end{equation}
Now, in view of the relations $\nu_1^m=n+2-\nu_{m+1}^n$ and $\nu_k^n=n+2-\nu_1^{k-1}$
we get the inequalities in the left sides in \eqref{abc1}.

Finally, let us verify the left inequality in \eqref{abc3}. In view of \eqref{abc5} we have
$$\nu_k^m=n+2-\nu_1^{k-1}-\nu_{m+1}^n \ge n+2 - (k+1)-(n-m+2) = m-k-1.$$
\end{proof}
From the particular case $m=k+2$ of \eqref{abc3} we have that
\begin{equation*}\label{abc6}1\le \nu_k^{k+2}\le 5,\ k=1,\ldots, n.\end{equation*}
Therefore, if in an exact problem $\left\{{{\bar\triangle}},\mathcal A\right\}$ some three successive triangles do not contain nodes then the problem is not poised.

Also, in view of \eqref{abc1}, we have that an exact problem is not poised if there are no nodes in the first or the last triangles of the strip.

In the last subsection we consider the case when there are no nodes in two successive triangles which do not include the first or the last triangles of the strip.

\subsection{Reduction "$0+0$"}

\begin{theorem} \label{00} Given an exact problem $\left\{{{\bar\triangle}},\mathcal A\right\},$ where the strip consists of $n$ triangles. Suppose that some two successive triangles $\Delta_k$ and $\Delta_{k+1},$ where $2\le k\le n-2,$ do not contain nodes. Then the problem $\left\{{{\bar\triangle}},\mathcal A\right\}$ is poised if and only if
the both following two reduced problems
\begin{equation}\label{3r}
\left\{{{\bar\triangle},\mathcal A}\right\}_{1}^{k-1},\quad \left\{{{\bar\triangle}},{\mathcal A}\right\}_{k+2}^n,
\end{equation}
are exact and poised.
\end{theorem}
\begin{proof}
Let us assume that the both reduced problems in \eqref{3r} are poised. Notice that then the problem $\left\{{{\bar\triangle}},\mathcal A\right\}$ is exact. Now let us prove, by following Proposition \ref{0p1}, ii), that the problem $\left\{{{\bar\triangle}},\mathcal A\right\}$ is poised, too.
Thus, assume that $s\in S({\bar\triangle}),\ s\vert_{\mathcal A}=0.$
From here we get that $s\vert_{{\mathcal A}_{1}^{k-1}}=0$ and $s\vert_{{\mathcal A}_{k+2}^n}=0.$
Since the two subproblems in \eqref{3r} are poised, we conclude that $s$ vanishes on the triangles
$\Delta_i,\ i=1,\ldots,k-1$ and
$\Delta_i,\ i=k+2,k+3,\ldots,n.$
Therefore $s$ vanishes also at the left side of the triangle $\Delta_{k}$ and at the right side of the triangle $\Delta_{k+1}.$
In particular $s$ vanishes at all the vertices of triangles $\Delta_{k}$ and $\Delta_{k+1}.$ Thus it vanishes on these two triangles too.
Hence $s=0.$

Next let us assume that the problem $\left\{{{\bar\triangle}},\mathcal A\right\}$ is poised and prove that the both reduced problems in \eqref{3r} are poised.

Let us show first that the both reduced problems in \eqref{1r} are exact.
There are $k-1$ and $n-k-1$ triangles in the reduced problems \eqref{3r}, respectively.
Therefore for exactness we need $k+1$ and $n-k+1$ nodes, respectively, altogether $n+2$ nodes.
 Now, assume by way of contradiction that a subproblem in \eqref{3r} is not exact. Then a subproblem is overdetermined and the other is underdetermined. Therefore, by Theorem \ref{main}, i), the problem is not poised, which is a contradiction.

Finally, let us show that both the reduced problems in \eqref{3r} are poised.
Assume by way of contradiction that one of them is not poised.
Then again, by Theorem \ref{main}, i), the problem is not poised, which is a contradiction.
 \end{proof}

\noindent Hayk Avdalyan, Hakop Hakopian \vspace{2mm}

\noindent{Department of Informatics and Applied Mathematics\\
Yerevan State University\\
A. Manukyan St. 1\\
0025 Yerevan, Armenia}

\vspace{1mm}

\noindent E-mails: avdalyanhayk@gmail.com, hakop@ysu.am

\end{document}